\documentclass[11pt]{amsart}
\usepackage{amsfonts,latexsym,amsthm,amssymb,graphicx}
\usepackage[all]{xy}
\DeclareFontFamily{OT1}{rsfs}{}
\DeclareFontShape{OT1}{rsfs}{n}{it}{<-> rsfs10}{}
\DeclareMathAlphabet{\mathscr}{OT1}{rsfs}{n}{it}

\newtheorem{theorem}{Theorem}[section]
\newtheorem{lemma}[theorem]{Lemma}
\newtheorem{corol}[theorem]{Corollary}
\newtheorem{prop}[theorem]{Proposition}

\newtheorem{conj}{Conjecture}

{\theoremstyle{definition} \newtheorem{defin}[theorem]{Definition}}
{\theoremstyle{remark} \newtheorem{remark}[theorem]{Remark}
\newtheorem{example}[theorem]{Example}

\newcommand{\Abb}{{\mathbb{A}}}

\newcommand{\Fbb}{{\mathbb{F}}}
\newcommand{\Lbb}{{\mathbb{L}}}
\newcommand{\Pbb}{{\mathbb{P}}}
\newcommand{\Qbb}{{\mathbb{Q}}}
\newcommand{\Zbb}{{\mathbb{Z}}}

\title{Stable birational equivalence and geometric Chevalley-Warning}
\author{Xia Liao}
\address{
Mathematics Department, 
Florida State University,
Tallahassee FL 32306, U.S.A.
}
\email{xliao@math.fsu.edu}
\begin{document}
\maketitle

\begin{abstract}
We propose a `geometric Chevalley-Warning' conjecture, that is a motivic extension of the Chevalley-Warning theorem in number theory. It is equivalent to a particular case of a recent conjecture of F. Brown and O.Schnetz. In this paper, we show the conjecture is true for linear hyperplane arrangements, quadratic and singular cubic hypersurfaces of any dimension, and cubic surfaces in $\Pbb^3$. The last section is devoted to verifying the conjecture for certain special kinds of hypersurfaces of any dimension. As a by-product, we obtain information on the Grothendieck classes of the affine `Potts model' hypersurfaces considered in \cite{aluffimarcolli1}.
\end{abstract}
\footnote{Keywords and phrases. Chevalley-Warning theorem, rationality, Grothendieck group.}
\footnote{2010 Mathematics Subject Classification. 14E08, 14N25, 14Q10.}

\section{introduction}

Let $\Fbb_q$ be the finite field of $q$ elements with $q$ a prime power. The Chevalley-Warning theorem states that the number of solutions in $\Fbb_q$ of a system of polynomial equations with $n$ variables is divisible by $q$, provided that the sum of the degrees of these polynomials is less than $n$. 
In \cite{brownschnetz}, \S3.3, F.~Brown and O.~Schnetz conjecture that a similar
statement holds in the Grothendieck ring of varieties over a $C_1$ field~$k$. They conjecture that the class of an affine $k$-variety defined by equations satisfying the same degree condition should be a multiple of the class $\Lbb$ of the affine line $\Abb^1_k$. Even the
`geometric' case, i.e., when $k$ is an algebraically closed field, appears to be open. We propose the following variant of
this conjecture: 

\begin{conj}[Geometric Chevalley-Warning]\label{CW}
Let $f_1,\ldots, f_l$ be homogeneous polynomials in $k[x_0,\ldots, x_n]$ such that $\sum_{i=1}^{l}deg(f_i)<n+1$, where $k$ is an algebraically closed field of characteristic $0$. Then $[Z(f_1,\ldots, f_l)] \equiv 1 (mod \ \Lbb)$ in $K_0(Var_k)$,
where $Z(f_1,\dots,f_l)$ denotes the set of zeros of $f_1,\dots,f_l$ in $\Pbb^n$.
\end{conj}

Over a field $k$ as in this statement, Conjecture \ref{CW} is equivalent to the conjecture of Brown and Schnetz. Indeed, let $X=Z(f_1,\dots,f_l)\subseteq \Pbb^n$ , then $[X]\cdot (\Lbb-1)+1$ is the class of the zero-locus of
$f_1,\dots,f_l$ in $\Abb^{n+1}$; the Brown-Schnetz conjecture would imply that this class is 
$\equiv 0\mod \Lbb$, and this is equivalent to $[X]\equiv 1 \mod \Lbb$. 

In this paper, we show 
that Conjecture~\ref{CW}
is true 
for 
hyperplane arrangements, 
for quadratic hypersurfaces of any dimension,  
for cubic surfaces in $\Pbb^3$,
and for singular cubic hypersurfaces in any dimension.
Along the way, we also establish a 
result which settles some special cases of the conjecture in higher dimensions. The hypothesis that $k$ is an algebraically closed field of characteristic zero is used in our proofs, and it is underlying the contextual remarks that follow in this introduction.

We note that the statement of Conjecture~\ref{CW} (for any $l$) is equivalent
to the case of hypersurfaces ($l=1$). Indeed,
we have the equality $[Z(f_1,f_2)]=[Z(f_1)]+[Z(f_2)]-[Z(f_1f_2)]$ in $K_0(Var)$, and the 
condition on degrees is satisfied by polynomials on one side of the equation whenever 
it is satisfied by polynomials on the other side of the equation. It follows that the 
conjecture is true for $Z(f_1,f_2)$ as long as it is true for the hypersurfaces $Z(f_1)$, 
$Z(f_2)$, and $Z(f_1f_2)$.  The same type of considerations applies to the zero set of 
any finite numbers of polynomial equations. 

When the variety in consideration is a hypersurface, the condition on degree asked 
by the geometric Chevalley-Warning conjecture becomes $\deg(f)<n+1$. This condition 
is reminiscent of results concerning the ``weakened rationality" of varieties. 
Recall that a variety is rationally chain connected if two general points on the variety can be 
joined by a chain of rational curves. It is known that a smooth hypersurface of degree $d$ in 
$\Pbb^n$ is rationally chain connected if and only if $d<n+1$ \cite{MR1158625}. Moreover, 
if we fix the degree of the hypersurface, and make the dimension of the ambient 
projective space large enough, then it is proved that a general such hypersurface is 
unirational \cite{paranjapesrinivas}. 

Introducing the notion of $\Lbb$-rationality, Conjecture~\ref{CW} admits an equivalent reformulation. 

\begin{defin} \cite{MR2775124}
A variety is $\Lbb$-rational if its class in $K_0(Var)$ is 1 modulo $\Lbb$.
\end{defin}

\begin{conj}\label{lrat}
Every hypersurface of degree $<n+1$ in $\Pbb^n$ is $\Lbb$-rational.
\end{conj}

Conjecture~\ref{lrat} postulates that $\Lbb$-rationality behaves in a sense as rational chain connectedness
does. For instance, we know that neither cubic nor quartic smooth 
threefolds in $\Pbb^4$ are {\em rational,\/} \cite{MR0291172}, \cite{MR0302652}, while they would be both $\Lbb$-rational and rationally chain connected according to Conjecture~\ref{lrat} and previous discussion.

The notion of $\Lbb$-rationality is motivated by stable rationality. We recall that two nonsingular irreducible varieties $X$ and $Y$ are ``stably birational'' if 
$X \times \Pbb^k$ is birational to $Y \times \Pbb^l$ for some $k$ and $l$.
We say that a nonsingular, complete irreducible variety is `stably rational'
if it is stably birational to projective space. For nonsingular varieties, $\Lbb$-rationality and stable rationality are equivalent. The argument is the following. Recall that
the ideal generated by $\Lbb$ in $K_0(Var)$ has a concrete meaning in stably 
birational geometry. 
Denote by $\Zbb[SB]$ the monoid ring generated by stably birational classes of smooth
complete irreducible varieties. Then M.~Larsen and V.~Lunts prove
in \cite{MR1996804} that there exists a surjective homomorphism $\Psi_{SB} : 
K_0(Var) \rightarrow \Zbb[SB]$, 
mapping the class of a smooth complete variety in $K_0(Var)$ to its class in $\Zbb[SB]$,
and the kernel of this homomorphism is precisely $(\Lbb)$. 
Thus, a smooth projective variety is stably birational to projective space precisely
when its class in $K_0(Var)$ is $1$ modulo $\Lbb$. 

The reader should note that e.g., every {\em cone\/} is $\Lbb$-rational; cf.~Lemma~\ref{cone}. Also, according to the result we recalled above, a smooth projective rational variety is $\Lbb$-rational. However, {\em singular\/} rational varieties may well not be $\Lbb$-rational. For example, if the normalization morphism of an irreducible rational curve is not set theoretically injective, then the curve itself is not $\Lbb$-rational.  Thus, `most' singular rational curves are not $\Lbb$-rational. Examples in higher dimension may be obtained by applying Lemma~\ref{special}. While all varieties considered in this paper are ruled or rational, it is by no means obvious a priori that they should be $\Lbb$-rational as prescribed by conjecture \ref{CW}, and as we prove below.

We wrap up this discussion by noting that rationally chain connectedness admits a
description analogous to the description of stable rationality we just recalled.
In \cite{kahnsujatha}, B.~Kahn and R.~Sujatha construct a category of pure birational 
motives by localizing the category of pure motives with respect to certain classes of 
birational morphisms. They prove (\cite{kahnsujatha}, \S3.1) that if $X$ is a 
rationally chain connected smooth projective $F$-variety, then $h^\circ(X)=1$ in 
$Mot_{rat}^\circ(F,\Qbb)$. Thus, $h^\circ(X)$ plays for rational connectedness
a role analogous to the role played by the class of X in $\Zbb[SB]\cong K_0(Var)/(\Lbb)$ for stable 
rationality.

\section{A few simple cases of the conjecture}\label{simplecases}

In this section we verify that the conjecture is true 
when the degrees of the homogeneous polynomials defining the variety are low. 
Namely, we will show the following results:

\begin{prop}\label{lineareqs}
If $X$ is the union of n or fewer hyperplanes in $\Pbb^n$, then $X$ is $\Lbb$-rational.
\end{prop}

%
%

\begin{prop}\label{quadratic}
Any quadratic hypersurface in $\Pbb^n \ (n>1)$ is $\Lbb$-rational.
\end{prop}

Proposition \ref{lineareqs} can be proved in a way similar to the reduction of the varieties in Conjecture ~\ref{CW} to hypersurfaces. In fact, the equation of $X$ can be written as $f_1\ldots f_l$ where $l$ is the numbers of hyperplanes and the $f_i$'s are all linear equations. Then $[X]=Z[f_1\ldots f_{l-1}]+Z[f_l]-Z[f_1\ldots f_{l-1}, f_l]$. The notation $Z[\ldots]$ indicates the set of common zeros of the equations appearing in the bracket, separated by commas, as mentioned in Conjecture ~\ref{CW}. The last term is the class of the union of $l-1$ hyperplanes in $\Pbb^{n-1}$. By induction, all terms on the right side of the equation are equivalent to 1 modulo $\Lbb$, so is $[X]$.
 
To prove Proposition ~\ref{quadratic}, we observe that any singular quadratic hypersurface is a cone. According to the following lemma and its corollary, the singular case can be taken care of generally, and we are left to consider the class of a nonsingular quadratic hypersurface.

\begin{lemma}\label{join}
Let Z be the join of varieties X and Y, which is obtained by connecting pairs of points from X and Y by $\Pbb^1$ and assuming these rational curves only meet at points of X or Y. If either X or Y is $\Lbb$-rational, then Z is also $\Lbb$-rational.
\end{lemma}

\begin{proof}
Taking out $X$ and $Y$ from the variety $Z$, we get a bundle over $X\times Y$ whose fiber is $\Pbb^1$ taken out 2 points. Thus the class of $Z$ in $K_0(Var)$ is $[X]\cdot [Y]\cdot (\Lbb -1)+[X]+[Y]=[X]\cdot [Y]\cdot \Lbb -([X]-1)\cdot([Y]-1)+1$.
\end{proof}
\begin{corol}\label{cone}
If the projective variety $X'\subset \Pbb^m$ is a cone over another projective variety $X \subset \Pbb^n$, $n<m$, then $X'$ is $\Lbb$-rational. 
\end{corol}
 
\begin{remark}
The union of $n$ or fewer hyperplanes in $\Pbb^n$ is a cone (over an arbitrary point in the set of the intersection). Thus we get a new proof of Propostion ~\ref{lineareqs}.
\end{remark} 

A special case of the result by Larsen-Lunts mentioned in the introduction gives an effective treatment of the Chevalley-Warning problem for nonsingular quadratic hypersurfaces, as 

\begin{lemma}~\label{rational}
Every rational smooth complete variety is $\Lbb$-rational.
\end{lemma}

\begin{proof}
A rational smooth complete variety has the same stable birational class as a point. Thus the difference of its class in $K_0(Var)$ and 1 is in the ideal generated by $\Lbb$. \cite{MR1996804}
\end{proof}

The previous lemmas settle the Chevalley-Warning problem for quadratic hypersurfaces. However, we can give another proof for the nonsingular one avoiding using Lemma~\ref{rational}. Let $Q$ be a nonsingular quadratic hypersurface. The projection from the ``north pole'' gives a nice birational map between $Q$ and $\Abb^n$ which allows us to stratify $Q$.

\begin{proof}[Proof of Proposition~\ref{quadratic} in the nonsingular case]

First, let's fix some notations. Let $Q_n$ be the nonsingular quadratic hypersurface in 
$ \Pbb^{n+1} $ defined by the equation $ X^2_0+X^2_1+\ldots+X^2_{n+1}=0 $ 
and $Y_n$ be the affine variety defined by $ \sum_{i=1}^{n+1} y^2_i=1 $ in 
$ \Abb^{n+1} $.
  
Projecting from the point $ P=(0,0,\ldots,1) $, we can establish a birational map between $ Y_n $ and $ \Abb^n $. The formula is given by:
  \begin{equation*}
  x_i=-\frac {y_i}{y_{n+1}-1}
  \end{equation*}
  Here, the $x_i$'s are coordinates of the affine space $ \Abb^n $.
  
  The inverse rational map from $ \Abb^n $ to $Q_n$ is given by the formula:
  \begin{equation*}
  y_i=\frac {2x_i}{\sum_{i=1}^{n} x^2_i +1} \quad
  y_{n+1}=\frac {\sum_{i=1}^{n} x^2_i -1}{\sum_{i=1}^{n} x^2_i +1}
  \end{equation*}
  
  From this description, it is easy to see the closed set $Z(\sum x^2_i +1) $ is not in the image of the polar projection. Then we have the following relation in $ K_0(Var) $:
  \begin{equation}
  [Y_n]-[Z(\sum_{i=1}^n y^2_i)]=[\Abb^{n}]-[Y_{n-1}]
  \end{equation}
  In addition to this relation, we also have the trivial relation
  \begin{equation}
  [Q_n]=[Q_{n-1}]+[Y_n]
  \end{equation}
  Because the variety $ Z(\sum_{i=1}^n y^2_i) $ is a cone, by the proof of ~\ref{cone}, we get $ Z(\sum_{i=1}^n y^2_i)$ is equal to $ 1+(\Lbb-1)[Q_{n-2}] $ in $ K_0(Var) $. So we can replace our first equation by:
  \begin{equation}
  [Y_n]-(1+(\Lbb-1)[Q_{n-2}])=[\Abb^n]-[Y_{n-1}]
  \end{equation}
  With the last two equations and the simple cases $ [Q_1]=\Lbb+1 $, $ [Y_0]=2 $, $ [Y_1]=\Lbb-1 $, we can conclude by an induction on dimension that $ [Q_n] \equiv 1\ (mod \ \Lbb) $ when $ n>0 $ and $ [Y_n] \equiv 0 \ (mod \ \Lbb) $ when $n>1$.
\end{proof}

\section{cubic hypersurfaces}\label{cubic}
The next easiest case to consider is the variety defined by a cubic equation in $ \Pbb^3 $. We have the following theorem:
\begin{theorem}\label{cubic}
Any cubic surface in $ \Pbb^3 $ is $\Lbb$-rational.
\end{theorem}

In fact, we can prove something more:
\begin{theorem}\label{scubic}
Any singular cubic hypersurface in $\Pbb^n \ (n\geqslant 3)$ is $\Lbb$-rational.
\end{theorem} 

The following lemma helps to analyze singular cubic hypersurfaces.

\begin{lemma}\label{special}
Let $X\subseteq \Pbb^n$ have equation
$F=x_n f_k(x_0,\ldots,x_{n-1})+f_{k+1}(x_0,\ldots,x_{n-1}) =0$ where $f_k$ and 
$f_{k+1}$ are homogeneous polynomials of degree $k$ and $k+1$ respectively. 
$X$ is $\Lbb$-rational if and only if the variety in $\Pbb^{n-1}$ defined by $f_k= 0$ is $\Lbb$-rational.
\end{lemma} 

\begin{proof}[Proof of Lemma ~\ref{special}]
On the hypersurface $Z(F)$, the equation $f_k=0$ defines a cone over the point 
$(0,\ldots,1)$. By Corollary~\ref{cone}, this subvariety of the hypersurface $Z(F)$ is 
$\Lbb$-rational. On the other hand, we have the 
isomorphism between the affine open set $f_k \neq 0$ in $\Pbb^{n-1}$ and $Z(F)-Z(f_k)$ provided by $(x_0,\ldots,x_{n-1}) 
\rightarrow (x_0,\ldots,x_{n-1},-\frac{f_{k+1}}{f_k})$. We see the $\Lbb$-rationality of $X$ is equivalent to the condition that the class of the affine open set $f_k \neq 0$ is in the ideal generated by $\Lbb$. This happens if and only if the hypersurface $f_k=0$ in $\Pbb^{n-1}$ is $\Lbb$-rational.  
\end{proof}

\begin{proof}[Proof of theorem~\ref{scubic}]
For a singular cubic hypersurface, we can assume one of its singular point is $(0:\ldots:0:1)$ by a change of projective coordinates. We keep using the notation of the previous lemma. The equation of the cubic surface is written as $x_n f_2(x_0,\ldots,x_{n-1})+f_3(x_0,\ldots,x_{n-1})=0$ or $f_3(x_0,\ldots,x_{n-1})=0$, depending on the singular point is a double point or a triple point. Then the $\Lbb$-rationality of the singular cubic hypersurface follows immediately from lemma ~\ref{special}, proposition ~\ref{quadratic} or corollary ~\ref{cone}.
\end{proof}

\begin{proof}[Proof of theorem~\ref{cubic}]
A nonsingular cubic surface arises from the blow-up of $\Pbb^2$ at 6 general points. 
In particular, it is rational. The $\Lbb$-rationality follows from Lemma~\ref{rational}.
The singular case has been done by theorem~\ref{scubic}.
\end{proof}

\begin{remark}
One can also approach theorem ~\ref{cubic} by the classification of cubic surfaces given in \cite{MR533323}. That is, one can directly compute their class in $K_0(Var)$ according to the standard equations given in this reference. This however leads to a lengthy computation.
\end{remark}

\begin{remark}
The criterion we derived in lemma \ref{special} can be applied to give another proof of proposition \ref{quadratic}. Since as long as the rank of the quadratic form is greater or equal to 2, the equation can be written as $F=x_0 x_1 + x_2^2 + \ldots$.
\end{remark} 

\begin{corol}\label{quartic}
If a singular quartic hypersurface in $\Pbb^4$ has a triple point, then it is $\Lbb$-rational. 
\end{corol}

\begin{proof}
Assuming the triple point is $(0:0:0:0:1)$, then the equation of the quartic hypersuface can be written as $F=x_4 f_3(x_0,x_1,x_2,x_3) + g_4(x_0,x_1,x_2,x_3)$. The $\Lbb$-rationality follows immediately from lemma~\ref{special} and theorem~\ref{cubic}.
\end{proof}

\section{$\Lbb$-rationality of higher dimensional varieties}

\begin{theorem}
If the equation of a hypersurface of degree $n$ in $\Pbb^m \ (m\geqslant n\geqslant4)$ can be written as $F=x_n \ldots x_4 f_3(x_0, x_1,x_2,x_3) + \sum_{i=5}^{n}x_n\ldots x_{i} g_{i-1}(x_0,\ldots,x_{i-2})+g_n(x_0,\ldots,x_{n-1})$, then this hypersurface is $\Lbb$-rational.
\end{theorem}
\begin{proof}
When $m>n$, not all coordinates of $\Pbb^m$ appear in $F$. In this case $F$ defines a cone in $\Pbb^m$ and the $\Lbb$-rationality of this hypersurface follows from lemma~\ref{cone}. When $m=n$, Rewrite the polynomial as
\begin{multline*}
F=x_n[x_{n-1}\ldots x_4 f_3(x_0,x_1,x_2,x_3) + \sum_{i=5}^{n-1}x_{n-1}\ldots x_{i}g_{i-1}(x_0,\ldots,x_{i-2})\\+g_{n-1}(x_0,\ldots,x_{n-2})] + g_n(x_0,\ldots,x_{n-1}).
\end{multline*}
Then the proof follows by induction, lemma~\ref{special} and corollary~\ref{quartic}. 
\end{proof}

\begin{theorem}
If the equation of the hypersurface of degree at most $n$ in $\Pbb^n \ (n\geqslant 4)$ has degree 1 in all variables except at most 4 variables, then this hypersurface is $\Lbb$-rational.
\end{theorem}

\begin{proof}
Suppose the 4 possibly non-linear variables are $x_0,x_1,x_2,$ and $x_3$. Proceed the proof by an induction on $n$. When $n=4$, the equation of the hypersurface is either $x_4 f_k(x_0,x_1,x_2,x_3)+f_{k+1}(x_0,x_1,x_2,x_3)=0 \ (k\leqslant 3)$ or $f_k(x_0,x_1,x_2,x_3)=0 \ (k\leqslant 4)$. Since we have checked the $\Lbb$-rationality of cubic surfaces, quadratic hypersurfaces, the $\Lbb$-rationality of such hypersurfaces are guaranteed by lemma~\ref{special} or corollary~\ref{cone}. When $n>4$, if the equation of the hypersurface is written only in terms of $x_0,x_1,x_2,x_3$, then it is a cone. Otherwise, let $x_n$ be one of its linear variables. Then the equation of the hypersurface is $x_n f_k(x_0,\ldots, x_{n-1})+f_{k+1}(x_0,\ldots,x_{n-1})$ where $f_k$ is again linear in all variables except at most 4 variables. So the $\Lbb$-rationality follows from lemma~\ref{special} and the induction hypothesis. 
\end{proof}

\begin{remark}
Theorem~4.2 generalizes Corollary 3.3 from \cite{MR2775124}, since the equations of the ``graph hypersurfaces'' considered there are linear in all variables, they must be $\Lbb$-rational by Theorem~4.2.
\end{remark}

\begin{example}[Affine Potts model hypersurface]
The previous theorem also gives an easy way to calculate the modulo $\Lbb$ class of the affine Potts model hypersurfaces appearing in \cite{aluffimarcolli1} definition (2.2) equation (2.5). The equations for such affine hypersurfaces are
$$ Z_G(q,t)=\sum_{G' \subseteq G'}q^{k(G)} \prod_{e \in E(G')}t_e,$$
 where $G'$ is a subgraph of $G$, $k(G')$ and $E(G')$ are the number of connected components and the set of edges of the graph $G'$ respectively. Fix $q$ in this equation and denote by $n$ the number of edges of the graph $G$, then this equation defines the affine Potts model hypersurface in $\Abb^n$. Its class in $K_0(Var)$ is congruent to 1 modulo $\Lbb$ if $n$ is odd and congruent to $-1$ modulo $\Lbb$ when $n$ is even. 

The calculation goes as follows. These hypersurfaces are defined by inhomogoneous polynomials of degree $n$ in $\Abb^n$, linear in all variables, where $n$ is the number of the edges of the graph in consideration. Homogenize the equation and write it as $F=0$, then clearly $F$ is a homogeneous polynomial of degree $n$ linear in all variables except the variable $x_0$ introduced in homogenizing. Now the class of the affine Potts model hypersurface is $[Z(F)]-[Z(x_0, x_1x_2\ldots x_{n})]$. $Z(F)$ is $\Lbb$-rational by the previous theorem, and now we are left to calculate $[Z(x_0, x_1x_2\ldots x_{n})]$ which is the class of the union of $n$ hyperplanes in $\Pbb^{n-1}$.

Let $x_1,\ldots, x_{n}$ be the projective coordinates of $\Pbb^{n-1}$. Consider the complement of $Z(x_1x_2\ldots x_{n})$, which can be explicitly expressed as points $(x_1:\cdots :x_{n})$ such that all projective coordinates are nonzero. We see this affine open set is isomorphic to the $(n-1)$-fold cartesian product of $(\Abb^1 - \{pt\})$. So the class of the union of $n$ hyperplanes in $\Pbb^{n-1}$ equals $[\Pbb^{n-1}]- (\Lbb -1)^{n-1}$. Taking into account that $[\Pbb^{n-1}]= 1+\Lbb +\cdots +\Lbb^{n-1}$, we see this class is $1+(-1)^n$ modulo $\Lbb$. We conclude the class of the affine Potts model hypersurface is $\Lbb$-rational when $n$ is odd and congruent to $-1$ modulo $\Lbb$ when $n$ is even.  
\end{example} 

%

\bibliographystyle{alpha}
\bibliography{liaobib}
  
\end{document}